\documentclass[11pt]{article}
\usepackage{lineno,hyperref}
\usepackage{bbm}
\usepackage{float}
\usepackage{mathrsfs}
\usepackage{graphicx}
\usepackage[utf8]{inputenc}
\usepackage{amsmath}
\usepackage{amstext}
\usepackage{amsfonts}
\usepackage{amsthm}
\usepackage{amssymb}
\usepackage[bf,SL,BF]{subfigure}
\usepackage{subfigure}
\usepackage{caption}
\newcommand{\triple}{{\vert\kern-0.25ex\vert\kern-0.25ex\vert}}
\allowdisplaybreaks
\newtheorem{definition}{Definition}[section]
\newtheorem{theorem}[definition]{Theorem}
\newtheorem{lemma}[definition]{Lemma}

\newtheorem{assumption}[definition]{Assumption}

\newtheorem{example}[definition]{Example}

\begin{document}
\title{\bf  Convergence in probability of numerical solutions of a highly nonlinear delayed stochastic interest rate model}
\author
{{\bf Emmanuel Coffie \footnote{Corresponding author, Email: emmanuel.coffie@liverpool.ac.uk}}
\\[0.2cm]
Institute for Financial and Actuarial Mathematics, 
 \\[0.2cm]
University of Liverpool, Liverpool L69 7ZL, UK.}
\date{}
\maketitle

\begin{abstract}
We study a delayed stochastic interest rate model with superlinearly growing coefficients and develop novel analytical tools to investigate the properties of both the true solution and its truncated Euler–Maruyama (TEM) approximation. In particular, we prove that the true solution converges in probability to the truncated EM solution as the step size approaches zero. Furthermore, we illustrate the theoretical findings through numerical experiments and validate the convergence results using an efficient Monte Carlo simulation framework for the valuation of relevant financial quantities.


\end{abstract}

\section{Introduction}
Stochastic modelling of interest rates plays a fundamental role in the calibration and valuation of financial derivatives, particularly option contracts. A wide range of stochastic interest rate models has been developed in the literature to describe the evolution of interest rates over time. Among the most prominent is the Cox–Ingersoll–Ross (CIR) model, introduced by Cox, Ingersoll, and Ross in \cite{CIR}. The CIR model is governed by the stochastic differential equation (SDE)
\begin{align}\label{cir1}
dx(t) = \alpha(\mu - x(t)),dt + \sigma \sqrt{x(t)},dB(t),
\end{align}
Here $x=(x(t),t\geq0)$ denotes the interest rate with initial value $x(0)=x_0$, $\alpha,\mu,\sigma>0$ and $B=(B(t),t\ge 0)$ is a scalar Brownian motion. The CIR model is mean-reverting and due to its square root diffusion factor, it can also avoid possible negative rates. 
\par
We observe in SDE \eqref{cir1} that the volatility term $\sigma$ is assumed constant. However, as supported by empirical studies, volatility is not constant but exhibits empirical features widely known as volatility skews and smiles that are prevalent in option markets (see, e.g., \cite{Heston, Gatheral}). To provide an adequate description of the evolution of volatility skews and smiles, one of the main schools of thought proposed based on empirical findings is to model volatility as a delay variable. For instance, Arriojas et al. in \cite{Arriojas} proposed the delay Black-Scholes model where the drift depends on a delay variable and, also, the volatility term is a function of a delay variable. This model is described by the dynamics
\begin{equation}
dx(t)=\mu x(t-a)x(t)dt+g(x(t-b))x(t)dB(t)
\end{equation}
on $t\in[0,T]$ with initial value $\xi(u)\in [-\tau,0]$, where $\xi:\Omega\rightarrow C([-\tau,0];\mathbb{R})$, $\mu,a,b>0$, $\tau=\max(a,b)$, $x(t-a)$ and $x(t-b)$ denote delays in $x(t)$, $g:\mathbb{R}\rightarrow \mathbb{R}$ is a continuous function and $B=(B(t),t\ge 0)$ is a scalar Brownian motion. The authors showed that the delay Black-Scholes model maintains the no-arbitrage property and the completeness of the market with the correct volatility skews and smiles. Furthermore, Mao and Sabanis in \cite{sabanis} introduced the delay geometric Brownian motion described by the stochastic differential delay equation (SDDE)
\begin{equation} \label{cir2}
dx(t)=rx(t)dt+v(x(t-\tau))x(t)dB(t)
\end{equation}
on $t\ge 0$ with the initial value $\xi(u)$ on $u\in[-\tau,0]$, where $\tau,r>0$, $\xi:\Omega\rightarrow C([-\tau,0];\mathbb{R})$, $B=(B(t),t\ge 0)$ is a scalar Brownian motion, the volatility function $v$ depends on $x(t-\tau)$ and $x(t-\tau)$ denotes delay in $x(t)$. The authors studied the quantitative properties of this model where  $v$ is merely local Lipschitz continuous and bounded. The authors provided numerical evidence to justify that the system of type \eqref{cir2} is a rich alternative model for an asset price process in a complete market characterised by volatility skews and smiles. The reader may also consult (e.g., \cite{Fuke1,Kind}) for financial models with features of past dependency. 
\par
In recent years, more empirical studies have shown that the most successful continuous-time models for interest rates are those models that allow the volatility of interest changes to be highly sensitive to the level of the rates (see, e.g., \cite{Chan,Nowman}). This motivated Wu et al. in  \cite{Fuke2} to extend SDE \eqref{cir1} to include super-linear diffusion term described by the dynamics
\begin{align}\label{cir3}
dx(t)=\alpha(\mu-x(t))dt+\sigma x(t)^{\theta}dB(t)
\end{align}
on $t\ge 0$, where $\theta>1$.  One notable unique feature of SDE \eqref{cir3} is that the solution $x(t)$ is a highly sensitive mean-reverting process. The authors established the convergence in probability of the EM  solutions to the true solution. They justified the convergence result within Monte Carlo simulations to value the expected payoff of a bond and a barrier option. To further capture volatility skews and smiles and high nonlinearities in the rate, the authors in \cite{emma} extended the generalised Ait-Sahalia interest rate model to include a volatility as a function of a delay variable described by  SDDE
\begin{equation}\label{Ait}
  dx(t)=(\alpha_{-1}x(t)^{-1}-\alpha_{0}+\alpha_{1}x(t)-\alpha_{2}x(t)^{\gamma})dt+v(x(t-\tau))x(t)^{\theta}dB(t)
\end{equation}
on $t\geq 0$ with the initial value $\xi(u)$ on $u\in[-\tau,0]$, where $\alpha_{-1},\alpha_{0},\alpha_{1},\alpha_{2}, \tau>0$, $\xi:\Omega\rightarrow C([-\tau,0];\mathbb{R})$, $\gamma,\theta>1$ and $B=(B(t),t\ge 0)$ is a scalar Brownian motion, $v$ is a function of $x(t-\tau)$ and $x(t-\tau)$ denotes delay in $x(t)$. Under a monotone condition and the assumption that $v$ is local Lipschitz continuous and bounded, the authors proved the strong convergence of the TEM solutions to the true solution of SDDE \eqref{Ait} and justified that the strong convergence result can be applied to value a bond and a barrier option.
\par
Therefore, in order to account for high nonlinearities in the rates as well as the evolution of volatility skews and smiles, we consider it necessary to reformulate SDE \eqref{cir3} as SDDE with super-linearly growing drift and diffusion coefficients described by the dynamics
\begin{equation}\label{cir4}
dx(t)=\alpha(\mu-x(t)^{\gamma})dt+\sigma x(t-\tau)^{r}x(t)^{\theta}dB(t),
\end{equation}
on $t\ge 0$ with the initial value $\xi(u)$ on $u\in[-\tau,0]$, where $\tau>0$, $\xi:\Omega\rightarrow C([-\tau,0];\mathbb{R})$, $\theta, r>0$ and $\gamma>1$. We observe that both the drift factor $x^{\gamma}$  and the diffusion factor $x^{\theta}y^r$ of SDDE \eqref{cir4} are growing super-linearly and thus violate the global Lipschitz and linear growth conditions. This is further complicated by the presence of the unbounded delay variable $y$. Therefore, it can be very challenging to obtain the solution of SDDE \eqref{cir4} by an analytical closed-form formula. To the best of our knowledge, there exist no relevant literature for numerical treatment of SDDE \eqref{cir4} either in the strong sense or weak sense. In this case, we recognise the need to examine the feasibility of the system of SDDE \eqref{cir4} from  a viewpoint of financial applications. This motivates the need for an efficient numerical method with fast computational performance to estimate the solution. However, in most real-world applications, the explicit EM method is preferred to the implicit type due to its simple algebraic structure, low computational cost and acceptable convergence rate. It is well-known in \cite{weakEM} that the explicit EM scheme diverges in the strong mean-square sense at finite point for SDEs with super-linearly growing coefficient terms. In this work, we aim to construct a variant of the truncated EM method developed in \cite{mao3} to estimate the true solution of SDDE \eqref{cir4} and show that the TEM solutions converge to the true solution in probability when the step size is sufficiently small. The remainder of the paper is organised as follows: We explore mathematical notations in Section \ref{sect2}. In Section \ref{sect3}, we study properties of the true solution of SDDE \eqref{cir4}. We construct the truncated EM techniques to approximate SDDE \eqref{cir4} and study properties of the TEM solutions in Section \ref{sect4}. In Section \ref{sect5},  we show that the TEM solutions converge to the true solution of SDDE \eqref{cir4} in probability. We also provide illustrative numerical examples to support the convergence result and justify the result via efficient use of the Monte Carlo method to value a bond and a lookback put option in this section.
\section{ Mathematical preliminaries}\label{sect2}
Throughout this paper, unless specified otherwise, we employ the following notation. Let $\{ \Omega,\mathcal{F},\mathbb{P}\}$ be a complete probability space with filtration $\{ \mathcal{F}_t\}_{t\geq 0}$ satisfying the usual conditions (i.e., it is increasing and right continuous while $\mathcal{F}_0$ contains all $\mathbb{P}$ null sets), and let $\mathbb{E}$ denote the expectation corresponding to $\mathbb{P}$. Let $B=(B(t), t\geq 0)$, be a scalar Brownian motion defined on the above probability space. If $a,b$ are real numbers, then $a\vee b$ denotes the maximum of $a$ and $b$, and $a\wedge b$ denotes the minimum of $a$ and $b$. Let $\mathbb{R}=(-\infty,\infty)$ and $\mathbb{R}_+=(0,\infty)$. If $x\in \mathbb{R}$, then $\vert x\vert$ is the Euclidean norm. For $\tau >0$, let $C([-\tau,0];\mathbb{R}_+)$ denote the space of all continuous functions $\xi: [-\tau,0]\rightarrow \mathbb{R}_+$ with the norm $\|\xi\|=\sup_{-\tau\leq t\leq 0}\xi(t)$. For an empty set $\emptyset$, we set $\text{inf }\emptyset=\infty$. For a set $A$, we denote its indicator function by $1_A$. Let the following scalar dynamics
\begin{equation}\label{eq2:1}
dx(t)=f(x(t))dt+g(x(t),x(t-\tau))dB(t)
\end{equation}
with initial value $x(u)=\xi(u)\in C([-\tau,0];\mathbb{R}_+)$ denote equation of SDDE \eqref{cir4} such that $f(x)=\alpha(\mu-x^{\gamma})$ and $g(x,y)=\sigma x^{\theta}y^r$, for all $x,y\in \mathbb{R}_+$. Let $C^{2,1}(\mathbb{R}\times \mathbb{R}_+;\mathbb{R})$ be the family of all real-valued functions $V(x,t)$ defined on $\mathbb{R}\times \mathbb{R}_+$ such that $V(x,t)$ is twice continuously differentiable in $x$ and once in $t$.  For each $V\in C^{2,1}(\mathbb{R}\times \mathbb{R}_+;\mathbb{R})$, define the  operator $LV:\mathbb{R}\times \mathbb{R}\times \mathbb{R}_+\rightarrow \mathbb{R}$ by
\begin{equation}\label{eq2:2}
 LV(x,y,t)=V_t(x,t)+V_x(x,t)f(x)+\frac{1}{2}V_{xx}(x,t)g(x,y)^2
\end{equation}
for SDDE \eqref{cir4} associated with the function $V$, where $V_t(x,t)$ and $V_x(x,t)$  are first-order partial derivatives with respect to $t$ and $x$ respectively, and $V_{xx}(x,t)$, a second-order partial derivative with respect to $x$. With the operator $LV$ defined, then the It\^{o} formula yields
\begin{align}\label{eq2:3}
dV(x(t),t)=LV(x(t),x(t-\tau),t)dt+V_x(x(t),t)g(x(t),x(t-\tau))dB(t)
\end{align}
almost surely. We should emphasise that $LV$ is defined on $\mathbb{R}\times \mathbb{R}\times \mathbb{R}_+$ while $V$ is defined on $\mathbb{R}\times \mathbb{R}_+$. Moreover, we impose the following standing hypotheses.
\begin{assumption} \label{sec2:assump:1} 
The parameters of \textup{SDDE} \eqref{cir4} satisfy 
\begin{equation}
1+ \gamma > 2(r+\theta),
\end{equation}
where  $\theta, r>0$ and $\gamma>1$.
\end{assumption}
\begin{assumption}\label{sec2:assump:2} 
There exist constants $D>0$ and $\ell\in (0,1]$ such that for all $ -\tau\leq u \leq t \leq 0$, the initial function $\xi$ satisfies
\begin{equation}\label{eq:15}
  |\xi(t)-\xi(u)|\le D|t-u|^{\ell}.
\end{equation}
\end{assumption}

\noindent We introduce the following lemma for later use.

\begin{lemma}\label{lemma1}
Let Assumption \ref{sec2:assump:1} hold. For any $R>0$, there exists a constant $G_R>0$ such that the coefficients of the \textup{SDDE} \eqref{cir4} satisfy
\begin{align}
| f(x)-f(\bar{x}) | + | g(x,y)-g(\bar{x},\bar{y}) |
\le G_R \big( |x-\bar{x}| + |y-\bar{y}| \big)
\end{align}
for all $x,y,\bar{x},\bar{y}\in \mathbb{R}$ with 
$|x|\vee |\bar{x}|\vee |y|\vee |\bar{y}| \le R$.
\end{lemma}

\begin{proof}
We only prove the result for $x \le \bar{x}$ and $y \le \bar{y}$; the general case follows similarly. First observe that
\begin{align*}
&(f(x)-f(\bar{x})) +(g(x,y)-g(\bar{x},\bar{y})) \\
&= \alpha(\mu - x^{\gamma}) - \alpha(\mu - \bar{x}^{\gamma})
  + \sigma x^{\theta}y^r - \sigma \bar{x}^{\theta}\bar{y}^r \\
&= -\alpha (x^{\gamma} - \bar{x}^{\gamma})
   + \sigma (x^{\theta}y^r - \bar{x}^{\theta}\bar{y}^r).
\end{align*}
Using Young inequality, we have
\begin{align*}
x^{\theta}y^r \le x^{2\theta} + y^{2r}, 
\quad
\bar{x}^{\theta}\bar{y}^r \le \bar{x}^{2\theta} + \bar{y}^{2r}.
\end{align*}
Hence,
\begin{align*}
x^{\theta}y^r - \bar{x}^{\theta}\bar{y}^r
\le (x^{2\theta} - \bar{x}^{2\theta}) + (y^{2r} - \bar{y}^{2r}).
\end{align*}
It follows that
\begin{align*}
&(f(x)-f(\bar{x}))+(g(x,y)-g(\bar{x},\bar{y})) \\
&\le -\alpha (x^{\gamma} - \bar{x}^{\gamma})
+ \sigma (x^{2\theta} - \bar{x}^{2\theta})
+ \sigma (y^{2r} - \bar{y}^{2r}).
\end{align*}
Applying the mean value theorem to each term yields
\begin{align*}
|x^{\gamma} - \bar{x}^{\gamma}|
&\le \gamma \big(|x|^{\gamma-1} + |\bar{x}|^{\gamma-1}\big) |x-\bar{x}|, \\
|x^{2\theta} - \bar{x}^{2\theta}|
&\le 2\theta \big(|x|^{2\theta-1} + |\bar{x}|^{2\theta-1}\big) |x-\bar{x}|, \\
|y^{2r} - \bar{y}^{2r}|
&\le 2r \big(|y|^{2r-1} + |\bar{y}|^{2r-1}\big) |y-\bar{y}|.
\end{align*}
Therefore,
\begin{align*}
&| f(x)-f(\bar{x}) | + | g(x,y)-g(\bar{x},\bar{y}) | \\
&\le \Big[\alpha \gamma (|x|^{\gamma-1} + |\bar{x}|^{\gamma-1})
+ 2\theta\sigma (|x|^{2\theta-1} + |\bar{x}|^{2\theta-1}) \Big] |x-\bar{x}| \\
&\quad + 2r\sigma (|y|^{2r-1} + |\bar{y}|^{2r-1}) |y-\bar{y}|.
\end{align*}
Since $|x|,|\bar{x}|,|y|,|\bar{y}| \le R$, all coefficients are bounded. Hence there exists a constant $G_R>0$ such that
\begin{align*}
| f(x)-f(\bar{x}) | + | g(x,y)-g(\bar{x},\bar{y}) |
\le G_R (|x-\bar{x}| + |y-\bar{y}|).
\end{align*}
This completes the proof.
\end{proof}
\section{Properties of true solution}\label{sect3}
In this section, we study properties of the true solution to SDDE \eqref{cir4}. Since SDDE \eqref{cir4} is a financial model, it is a natural requirement to show that the solution is always positive.
\subsection{Existence of positive solution}
The following theorem shows that the solution of SDDE \eqref{cir4} is positive almost surely.
\begin{theorem}
Let Assumption \ref{sec2:assump:1} hold. Then for any given initial value 
\begin{equation}\label{initial1}
\{ x(u): -\tau\leq u \leq 0\}=\xi(u) \in C([-\tau,0]:\mathbb{R}_+),
\end{equation}
there exists a unique solution $x(t)$ to \textup{SDDE} \eqref{cir4} and $x(t)>0$ almost surely.
\end{theorem}

\begin{proof}
Since the coefficients of SDDE \eqref{cir4} satisfy local Lipschitz condition in $[-\tau,\infty)$, one can show by the standard truncation method that there exists a unique maximal local solution $x(t)$ on $[-\tau,\eta_{e})$ for any given initial value \eqref{initial1}, where $\eta_{e}$ is the explosion time (see \cite{Rassias}). Let  $k_0>0$ be sufficiently large such that
\begin{equation*}\label{initial2}
  \frac{1}{k_0}<\underset{-\tau\le u\leq0}{\min}\vert \xi(u)\vert\le \underset{-\tau\le t\le 0}{\max}\vert\xi(u)\vert<k_0.
\end{equation*}
For each integer $k\geq k_0$, we define the stopping time by
\begin{equation}\label{stoptime1}
\eta_k=\inf\{ t\in [0,\eta_{e}): x(t)\not\in [1/k,k]\}.
\end{equation}
We observe that $\eta_k$ is increasing  as $k\rightarrow \infty$. We set $\eta_{\infty}=\underset{k\rightarrow \infty}\lim \eta_k$, whence $\eta_{\infty}\leq \eta_e$ almost surely. In other words, we need to show that $ \eta_{\infty}=\infty$ almost surely to complete the proof. We define a $C^2$-function $V:\mathbb{R_+}\rightarrow \mathbb{R_+}$ by
\begin{equation}\label{sec2:eq1}
V(x)=x^{\beta}-1-\beta\log(x), 
\end{equation}
where $\beta\in(0,1)$. By applying the operator defined in \eqref{eq2:2} to \eqref{sec2:eq1}, we compute
\begin{align*}
 LV(x,y)&=\beta(x^{\beta-1}-x^{-1})\alpha(\mu-x^{\gamma})+\frac{\sigma^2}{2}(\beta(\beta-1)x^{\beta-2}+\beta x^{-2})x^{2\theta}y^{2r}\\
&= \alpha\mu\beta x^{\beta-1}-\alpha\mu\beta x^{-1}-\alpha\beta x^{\gamma+\beta-1}+\alpha\beta x^{\gamma-1}\\
&+\frac{\sigma^2}{2}\beta(\beta-1)x^{2\theta+\beta-2}y^{2r}+\frac{\sigma^2}{2}\beta x^{2\theta-2}y^{2r}.
\end{align*}
We note that for $\beta\in(0,1)$, $\gamma+\beta-1>2\theta+2r+\beta-2\Rightarrow \gamma+1>2(r+\theta)$. So by Assumption \ref{sec2:assump:1} and for $\beta\in(0,1)$, $-\alpha\beta x^{\gamma+\beta-1}$ leads and tends to $-\infty$ for large $x$. However, for small $x$, $-\alpha\mu\beta x^{-1}$ leads and also tends to $-\infty$. Hence, we can find a constant $K_0$ such that 
\begin{equation}\label{last}
LV(x,y)\le K_0. 
\end{equation}
For any $t_1\in [0,\tau]$, the It\^{o} formula yields
\begin{align*}
\mathbb{E}(V(x(\eta_k\wedge t_1)))&\le V(\xi(0))+\mathbb{E}\int_0^{\eta_k\wedge t_1}K_0ds\\
&\le V(\xi(0))+K_0\tau,
\end{align*}
for all $k\ge k_0$. Noting that
\begin{align*}
\mathbb{P}(\eta_k\le \tau)\le \frac{\mathbb{E}[V(x(\eta_k\wedge t_1))]}{V(k)\wedge V(1/k)}, 
\end{align*}
we have
\begin{equation*}
\mathbb{P}(\eta_k\le \tau)\le \frac{V(\xi(0))+K_0\tau}{V(k)\wedge V(1/k)}.
\end{equation*}
As $k\rightarrow \infty$, $ \mathbb{P}(\eta_k\leq \tau)\rightarrow 0 $ and hence, $ \eta_{\infty}>\tau$ almost surely. For $t_1\in [0,2\tau]$, the It\^{o} formula gives us
\begin{align*}
\mathbb{E}(V(x(\eta_k\wedge t_1)))&\le V(\xi(0))+\mathbb{E}\int_0^{\eta_k\wedge t_1}K_0ds\\
&\le V(\xi(0))+2K_0\tau,
\end{align*}
for all $k\ge k_0$. This also means that 
\begin{equation*}
\mathbb{P}(\eta_k\le 2\tau)\le \frac{V(\xi(0))+2K_0\tau}{V(k)\wedge V(1/k)}.
\end{equation*}
As $k\rightarrow \infty$, we get $ \eta_{\infty}>2\tau$ almost surely. Meanwhile, for $t_1\in [0,T]$, we also derive from the It\^{o} formula that 
\begin{align*}
\mathbb{E}(V(x(\eta_k\wedge t_1)))&\le V(\xi(0))+\mathbb{E}\int_0^{\eta_k\wedge t_1}K_0ds\\
&\le V(\xi(0))+K_0T,
\end{align*}
for all $k\ge k_0$. This also means that 
\begin{equation}\label{probability1}
\mathbb{P}(\eta_k\le T)\le \frac{V(\xi(0))+K_0T}{V(k)\wedge V(1/k)}.
\end{equation}
Since $V(k)\wedge V(1/k)\to\infty$ as $k\to\infty$, it follows that
\begin{align*}
\lim_{k\to\infty} \mathbb{P}(\eta_k\le T)=0.
\end{align*}
Noting that $\eta_k \uparrow \eta_\infty$, we obtain
\begin{align*}
\mathbb{P}(\eta_\infty \le T)
= \lim_{k\to\infty} \mathbb{P}(\eta_k \le T)
= 0.
\end{align*}
Since $T>0$ is arbitrary, we conclude that $\mathbb{P}(\eta_\infty < \infty)=0$, and hence $\eta_\infty=\infty$ almost surely. Therefore,
\begin{align*}
\mathbb{P}\big(0 < x(t) < \infty \text{ for all } t \in [0,T]\big)=1.
\end{align*}
\end{proof}
\subsection{Boundedness}
We also present the following useful result that is required to establish uniform boundedness of the true solution of SDDE \eqref{cir4}.
\begin{lemma}\label{bound}
Let Assumption \ref{sec2:assump:1} hold and $\beta\in(0,1)$. Then for any initial value $\xi(0)$, the solution $x(t)$ of \textup{SDDE} \eqref{cir4} fulfils 
\begin{equation*}
\mathbb{E}\Big[ x(t)^{\beta}-1-\beta\log(x(t)) \Big]\le \xi(0)^{\beta}-1-\beta\log(\xi(0))+\bar{K}_0
\end{equation*}
for all $t\ge 0$ and 
\begin{equation*}
\underset{t\rightarrow \infty}{\limsup}\mathbb{E}\Big[x(t)^{\beta}-1-\beta\log(x(t))\Big]\le \bar{K}_0
\end{equation*}
where $\bar{K}_0$ is a positive constant that does not depend on the initial value $\xi(0)$.
\end{lemma}

\begin{proof}
We define $V_1\in C^{2,1}(\mathbb{R}_+\times \mathbb{R}_+;\mathbb{R}_+)$ by $V_1(x,t)=e^tV(x)$, where $V(x)$ is  the same as \eqref{sec2:eq1}. We compute from the diffusion operator in \eqref{eq2:2} that
\begin{align*}
LV_1(x,y,t)&=e^t\Big[x^{\beta}-1-\beta\log(x)+\alpha\mu\beta x^{\beta-1}-\alpha\mu\beta x^{-1}-\alpha\beta x^{\gamma+\beta-1}+\alpha\beta x^{\gamma-1}\\
&+\frac{\sigma^2}{2}\beta(\beta-1)x^{2\theta+\beta-2}y^{2r}+\frac{\sigma^2}{2}\beta x^{2\theta-2}y^{2r}\Big]\\
&=e^t\Big[x^{\beta}-1-\beta\log(x)\Big]+e^t\Big[\alpha\mu\beta x^{\beta-1}-\alpha\mu\beta x^{-1}-\alpha\beta x^{\gamma+\beta-1}+\alpha\beta x^{\gamma-1}\\
&+\frac{\sigma^2}{2}\beta(\beta-1)x^{2\theta+\beta-2}y^{2r}+\frac{\sigma^2}{2}\beta x^{2\theta-2}y^{2r}\Big].
\end{align*}
Hence, by \eqref{last}, there exists a constant such that
\begin{align*}
LV_1(x,y,t)&\le e^t(V(x)+LV(x,y))\\
&\le e^t\bar{K}_0.
\end{align*}
Using the same stopping time as defined in \eqref{stoptime1}, we derive from the It\^{o} formula that
\begin{align*}
\mathbb{E}V_1(x(t\wedge \tau_k),t\wedge \tau_k)\le V_1(\xi(0),0)+\int_0^{t\wedge \tau_k} LV_1(x(s),x(t-\tau),s)ds.
\end{align*}
This implies that
\begin{align*}
e^{t\wedge \tau_k}V(x(t\wedge \tau_k))\le V(\xi(0))+\int_0^{t\wedge \tau_k}e^s(V(x(s))+ LV(x(s),x(t-\tau)))ds.
\end{align*}
By applying the Fatou lemma and setting $k\rightarrow \infty$, we now have
\begin{align*}
e^tV(x(t))\le V(\xi(0))+e^t\bar{K}_0.
\end{align*}
This also implies that
\begin{align*}
x(t)^{\beta}-1-\beta\log(x(t))\le \frac{V(\xi(0))}{e^t}+\bar{K}_0,
\end{align*}
which gives both assertions as required. 
\end{proof}
The following result reveals that the true solution of SDDE \eqref{cir4} will stay in a compact support with large probability.
\begin{theorem}
Let Assumption \ref{sec2:assump:1} hold and $\beta\in(0,1)$.  Then for any initial value $\xi(0)$ and $k>k_0$, there exists a constant $\bar{K}_0$ such that 
\begin{align*}
\mathbb{P}(1/k< x(t)< k)\ge 1-\epsilon
\end{align*}
for all $t\ge 0$,
\begin{align*}
\epsilon=\Big[\xi(0)^{\beta}-1-\beta\log(\xi(0))+\bar{K}_0\Big]\Big[
\frac{1}{(1/k)^{\beta}-1+\beta\log(k)} +\frac{1}{k^{\beta}-1-\beta\log(k)}\Big].
\end{align*}
\end{theorem}
\begin{proof}
We compute from Lemma \ref{bound} that for any $t\ge 0$, we have 
\begin{align*}
\mathbb{P}(x(t)&\le 1/k)\le \mathbb{E}\Big[1_{\{x(t)\le 1/k\}}\frac{x(t)^{\beta}-1-\beta\log(x(t))}{(1/k)^{\beta}-1+\beta\log(k)}    \Big]\\
&\le \frac{\xi(0)^{\beta}-1-\beta\log(\xi(0))+\bar{K}_0}{(1/k)^{\beta}-1+\beta\log(k)}.
\end{align*}
Similarly, we also obtain from Lemma \ref{bound} that for any $t\ge 0$
\begin{align*}
\mathbb{P}(x(t)&\ge k)\le \mathbb{E}\Big[1_{\{x(t)\ge k\}}\frac{x(t)^{\beta}-1-\beta\log(x(t))}{k^{\beta}-1-\beta\log(k)}    \Big]\\
&\le \frac{\xi(0)^{\beta}-1-\beta\log(\xi(0))+\bar{K}_0}{k^{\beta}-1-\beta\log(k)}.
\end{align*}
This implies that,
\begin{align*}
\mathbb{P}(1/k< x(t)<k)\ge 1- \Big[\xi(0)^{\beta}-1-\beta\log(\xi(0))+\bar{K}_0\Big]\Big[
\frac{1}{(1/k)^{\beta}-1+\beta\log(k)} +\frac{1}{k^{\beta}-1-\beta\log(k)}\Big],
\end{align*}
as required.
\end{proof}
\section{Numerical method}\label{sect4}
In this section, we develop truncated EM techniques to estimate the solution of SDDE \eqref{cir4}. Moreover, we establish some properties of the numerical solutions.
\subsection{The truncated EM method}
Before we construct the numerical method,  we need to extend the domain of SDDE \eqref{cir4} from $\mathbb{R}_+$ to $\mathbb{R}$. We should mention that this extension does not affect previous results in anyway. To define the truncated EM method, we choose a strictly increasing continuous function $z:\mathbb{R}_+\rightarrow \mathbb{R}_+$ such that $z(u)\rightarrow\infty$ as $u\rightarrow \infty$ and 
\begin{equation}\label{sec4:eq:1}
\sup_{\vert x\vert\vee\vert y\vert \le u}\Big(|f(x)|\vee g(x,y)\Big)\le z(u), 
\end{equation}
for all $u\ge 0$. Denote by $z^{-1}$ the inverse function of $z$ and we see that $z^{-1}$ is strictly increasing continuous function from $[z(0),\infty)$ to $\mathbb{R}_+$. We also choose a number $\Delta^*\in (0,1]$ and a strictly decreasing function 
$\psi:(0,\Delta^*]\rightarrow \mathbb{R}_+$ such that
\begin{equation}\label{sec4:eq:2}
\quad \psi(\Delta^*)\ge z(1), \lim_{\Delta \rightarrow 0}\psi(\Delta)=\infty \text{ and } \Delta^{1/4}\psi(\Delta)\le 1, \quad \forall \Delta \in (0,1).
\end{equation}
For a given step size $\Delta \in (0,\Delta^*)$, we then define the truncated functions by
\begin{align*}  
f_{\Delta}(x)&=
\begin{cases}
  f\Big(x\wedge z^{-1}(\psi(\Delta))\Big), & \mbox{if $x\geq 0$ }\\
  \alpha\mu,                             & \mbox{if $x<0$},
\end{cases}
\\
g_{\Delta}(x,y)&=
\begin{cases}
  g\Big(x\wedge z^{-1}(\psi(\Delta)),y\wedge z^{-1}(\psi(\Delta))\Big), & \mbox{if $x,y\geq 0$ }\\
  0,                             & \mbox{if $x,y<0$},
\end{cases}
\end{align*}
for all $x,y\in \mathbb{R}$. So for $x,y\in[0,z^{-1}(\psi(\Delta))]$, we observe that
\begin{align}\label{sec4:eq:3}
|f_{\Delta}(x)|\vee g_{\Delta}(x,y)\le z(z^{-1}(\psi(\Delta)))=\psi(\Delta),
\end{align}
for all $x,y\in \mathbb{R}$. That is, $f_{\Delta}$ and $g_{\Delta}$ are bounded by $\psi(\Delta)$ although $f$ and $g$ are unbounded. From now on, we let $T > 0$ be arbitrarily fixed. We also let the step size $\Delta\in (0,\Delta^*]$ be a fraction of $\tau$, that is, $\Delta=\frac{\tau}{M}$ for some integer $M>\tau$. We construct the discrete-time truncated EM approximation of SDDE \eqref{cir4} by defining $t_k=k\Delta$ for $-M\le k\le\infty$, setting $X_{\Delta}(t_k)=\xi(t_k)$ for $-M\le k\le 0$ and computing 
\begin{equation}\label{eq:27}
X_{\Delta}(t_{k+1})=X_{\Delta}(t_k)+f_{\Delta}(X_{\Delta}(t_k))\Delta+g_{\Delta}(X_{\Delta}(t_{k}),X_{\Delta}(t_{k-M}))\Delta B_k
\end{equation}
for $k\ge 0$, where $\Delta B_k=B(t_{k+1})-B(t_k)$ is an increment of the Brownian motion. We define the continuous-time truncated EM step process by
\begin{equation}\label{TEM1}
\bar{x}_{\Delta}(t)=\sum_{k=-M}^{\infty}X_{\Delta}(t_k)1_{[t_k,t_{k+1})}(t)
\end{equation}
where $1_{[t_k,t_{k+1})}$ is the indicator function on $[t_k,t_{k+1})$.  The  continuous-time continuous truncated EM process is defined by setting $x_{\Delta}(u)=\xi(u)$ for $u\in [-\tau,0]$ while for $t\geq 0$, we get
\begin{equation}\label{TEM2}
x_{\Delta}(t)=\xi(0)+\int_0^t f_{\Delta}(\bar{x}_{\Delta}(s))ds+\int_0^t g_{\Delta}(\bar{x}_{\Delta}(s),\bar{x}_{\Delta}(s-\tau))dB(s).
\end{equation}
We observe that $x_{\Delta}(t)$ is an It\^{o} process on $t\geq 0$ satisfying It\^{o} differential
\begin{equation}\label{eq:30}
dx_{\Delta}(t)= f_{\Delta}(\bar{x}_{\Delta}(t))dt+g_{\Delta}(\bar{x}_{\Delta}(t),\bar{x}_{\Delta}(t-\tau))dB(t).
\end{equation}
It is important to note that $x_{\Delta}(t_{k})=\bar{x}_{\Delta}(t_k)=X_{\Delta}(t_k)$ for $k\ge -M$.
\subsection{Properties of numerical solution}
The following lemma shows that the discrete-time process $\bar{x}_{\Delta}(t)$ and the continuous-time process $x_{\Delta}(t)$ are close to each other in the strong sense.
\begin{lemma}\label{sec5:eq:L1}
For any fixed $\Delta\in (0,\Delta^*]$ and $p\ge 2$, we have 
\begin{equation}
 \mathbb{E}\vert x_{\Delta}(t)-\bar{x}_{\Delta}(t)\vert^{p}\le c_p\Delta^{\frac{p}{2}}(\psi(\Delta))^p,
\end{equation}
for all $t\geq 0$, where $c_p$ is a generic constant that is dependent only on $p$.
\end{lemma}
\begin{proof}
Fix any $\Delta\in (0,\Delta^*]$ and $t\ge 0$. Then there is a unique integer $k\ge 0$ such that $t_k\le t\le t_{k+1}$. By elementary inequality and \eqref{sec4:eq:3}, we have
\begin{align*}
\mathbb{E}\vert x_{\Delta}(t)-\bar{x}_{\Delta}(t)\vert^p&=\mathbb{E}\vert x_{\Delta}(t)-\bar{x}_{\Delta}(t_k)\vert^p\\
&\le C_p\Big(\mathbb{E}\big\vert\int_{t_k}^tf_{\Delta}(\bar{x}_{\Delta}(s))ds\big\vert^p+\mathbb{E}\big\vert\int_{t_k}^tg_{\Delta}(\bar{x}_{\Delta}(s),\bar{x}_{\Delta}(s-\tau))dB(s)\big\vert^p\Big)\\
&\le C_p\Big(\Delta^{p-1}\mathbb{E}\int_{t_k}^t\vert f_{\Delta}(\bar{x}_{\Delta}(s))\vert^p ds+\Delta^{\frac{p-2}{2}}\mathbb{E}\int_{t_k}^t\vert g_{\Delta}(\bar{x}_{\Delta}(s),\bar{x}_{\Delta}(s-\tau))\vert^p ds\Big)\\
&\le C_p\Big(\Delta^{p-1}(\psi(\Delta))^p\Delta+\Delta^{\frac{p-2}{2}}(\psi(\Delta))^p\Delta\Big)\\
&\le C_p\Big(\Delta^{p}(\psi(\Delta))^p+\Delta^{\frac{p}{2}}(\psi(\Delta))^p\Big)\\
&\le c_p\Delta^{\frac{p}{2}}(\psi(\Delta))^p,
\end{align*}
where $c_p=C_p\vee 1$ and from \eqref{sec4:eq:2}, we obtain $\Delta^{\frac{p}{2}}(\psi(\Delta))^p\le \Delta^{\frac{p}{4}}$.
\end{proof}
The following lemma reveals the probability that the TEM solutions do not explode in finite time. 
\begin{lemma}\label{sub3}
Let Assumptions \ref{sec2:assump:1} and \ref{sec2:assump:2}  hold and $T>0$ be fixed. For any sufficiently large integer $k>0$, define the stopping time by
\begin{equation}\label{sec5:eq:tau}
\eta_{\Delta}=\inf\{t\in [0,T]:x_{\Delta}(t)\notin [1/k,k]\}.
\end{equation}
Then for any fixed $\Delta\in(0,\Delta^*]$, we have
\begin{equation}\label{probability2}
 \mathbb{P}(\eta_{\Delta}\leq T)\leq \frac{V(\xi(0))+K_1T+K_3D\Delta^{\ell}+c_p(K_2+K_3)\Delta^{1/2}\psi(\Delta)T}{V(1/k)\wedge V(k)},
\end{equation}
where  $K_1$, $K_2$ and $K_3$ are generic constants and $V$ is defined in \eqref{sec2:eq1}.
\end{lemma}

\begin{proof}
For $t_1\in[0,T]$, we apply the It\^{o} formula to \eqref{eq:30} to compute
\begin{align*}
&\mathbb{E}(V(x_{\Delta}(t\wedge \eta_{\Delta})))-V(\xi(0))\\
&=\mathbb{E}\int_{0}^{t_1\wedge \eta_{\Delta}}\Big(V_x(x_{\Delta}(s))f_{\Delta}(\bar{x}_{\Delta}(s))+\frac{1}{2}V_{xx}(x_{\Delta}(s))g_{\Delta}(\bar{x}_{\Delta}(s),\bar{x}_{\Delta}(s-\tau))^2\Big)ds\\
&\le\mathbb{E}\int_{0}^{\eta_{\Delta}\wedge t_1}\Big(V_x(x_{\Delta}(s))f_{\Delta}(x_{\Delta}(s))+\frac{1}{2}V_{xx}(x_{\Delta}(s))g_{\Delta}(x_{\Delta}(s),x_{\Delta}(s-\tau))^2\Big)ds\\
&+\mathbb{E}\int_{0}^{\eta_{\Delta}\wedge t_1}V_x(x_{\Delta}(s))\Big(f_{\Delta}(\bar{x}_{\Delta}(s))-f_{\Delta}(x_{\Delta}(s))\Big)ds\\
&+\mathbb{E}\int_{0}^{\eta_{\Delta}\wedge t_1}\frac{1}{2}V_{xx}(x_{\Delta}(s))\Big(g_{\Delta}(\bar{x}_{\Delta}(s),\bar{x}_{\Delta}(s-\tau))^2-g_{\Delta}(x_{\Delta}(s),x_{\Delta}(s-\tau))^2\Big)ds.
\end{align*}
By recalling the definition of the truncated functions in \eqref{sec4:eq:1}, we note that
\begin{equation}\label{sec5:eq:recall0}
f_{\Delta}(\cdot)=f(\cdot)\text{ and } g_{\Delta}(\cdot,\cdot)=g(\cdot,\cdot).
\end{equation}
Also, for  $s\in[0,\eta_{\Delta}\wedge t_1]$ with $x_{\Delta}(s),\bar{x}_{\Delta}(s),x_{\Delta}(s-\tau),\bar{x}_{\Delta}(s-\tau)\in[1/k,k]$, we observe that
\begin{equation}\label{sec5:eq:recall1}
g(\bar{x}_{\Delta}(s),\bar{x}_{\Delta}(s-\tau))\vee g(x_{\Delta}(s),x_{\Delta}(s-\tau))\le z(k).
\end{equation}
By Assumption \ref{sec2:assump:2}, we have 
\begin{align*}
&\mathbb{E}(V(x_{\Delta}(t\wedge \eta_{\Delta})))-V(\xi(0))\le K_1T+\mathbb{E}\int_{0}^{\eta_{\Delta}\wedge t_1}V_x(x_{\Delta}(s))\vert f(\bar{x}_{\Delta}(s))-f(x_{\Delta}(s))\vert ds\\
&+\mathbb{E}\int_{0}^{\eta_{\Delta}\wedge t_1}\frac{1}{2}V_{xx}(x_{\Delta}(s))\Big(g(\bar{x}_{\Delta}(s),\bar{x}_{\Delta}(s-\tau))^2-g(x_{\Delta}(s),x_{\Delta}(s-\tau))^2\Big)ds,
\end{align*}
where $LV(x_{\Delta}(s),x_{\Delta}(s-\tau))\le K_1$ for $s\in[0,t\wedge \eta_{\Delta}]$. By  an elementary inequality, we have
\begin{align*}
&\mathbb{E}(V(x_{\Delta}(t\wedge \eta_{\Delta})))-V(\xi(0))\le K_1T+\mathbb{E}\int_{0}^{\eta_{\Delta}\wedge t_1}V_x(x_{\Delta}(s))\vert f(\bar{x}_{\Delta}(s))-f(x_{\Delta}(s))\vert ds\\
&+\mathbb{E}\int_{0}^{\eta_{\Delta}\wedge t_1}\frac{1}{2}V_{xx}(x_{\Delta}(s))\Big(\vert g(\bar{x}_{\Delta}(s),\bar{x}_{\Delta}(s-\tau))-g(x_{\Delta}(s),x_{\Delta}(s-\tau))\vert\\
&\times \vert g(\bar{x}_{\Delta}(s),\bar{x}_{\Delta}(s-\tau))+g(x_{\Delta}(s),x_{\Delta}(s-\tau))\vert \Big) ds.
\end{align*}
By \eqref{sec4:eq:3}, \eqref{sec5:eq:recall0}, \eqref{sec5:eq:recall1} and Lemma \ref{lemma1}, we now have
\begin{align*}
\mathbb{E}(V(x_{\Delta}(t\wedge \eta_{\Delta})))&\le V(\xi(0))+K_1T+\mathbb{E}\int_{0}^{t_1\wedge\eta_{\Delta}}G_kV_x(x_{\Delta}(s))\vert \bar{x}_{\Delta}(s)-x_{\Delta}(s)\vert ds\\
&+\mathbb{E}\int_{0}^{\eta_{\Delta}\wedge t_1}(z(k))^2G_kV_{xx}(x_{\Delta}(s))\vert \bar{x}_{\Delta}(s)-x_{\Delta}(s)\vert ds\\
&+\mathbb{E}\int_{0}^{\eta_{\Delta}\wedge t_1}(z(k))^2G_kV_{xx}(x_{\Delta}(s))\vert \bar{x}_{\Delta}(s-\tau)-x_{\Delta}(s-\tau)\vert ds\\
&\le V(\xi(0))+K_1T+K_2\mathbb{E}\int_0^{\eta_{\Delta}\wedge t_1} \vert \bar{x}_{\Delta}(s)-x_{\Delta}(s)\vert ds\\
&+K_3\mathbb{E}\int_{0}^{\eta_{\Delta}\wedge t_1}\vert \bar{x}_{\Delta}(s-\tau)-x_{\Delta}(s-\tau)\vert ds,
\end{align*}
where
\begin{equation*}
K_2=\max_{1/k\le x\le k}\Big(G_kV_x(x)+(z(k))^2G_kV_{xx}(x)\Big)
\end{equation*}
and
\begin{equation*}
K_3=\max_{1/k\le x\le k}\Big((z(k))^2G_kV_{xx}(x)\Big).
\end{equation*}
So by Assumption \ref{sec2:assump:1} and Lemma \ref{sec5:eq:L1}, we get
\begin{align*}
\mathbb{E}(V(x_{\Delta}(t\wedge \eta_{\Delta})))&\le V(\xi(0))+K_1T+K_2\mathbb{E}\int_0^{\eta_{\Delta}\wedge t_1} \vert \bar{x}_{\Delta}(s)-x_{\Delta}(s)\vert ds\\
&+K_3\mathbb{E}\int_{-\tau}^{\eta_{\Delta}\wedge t_1}\vert \bar{x}_{\Delta}(s)-x_{\Delta}(s)\vert ds\\
&\le V(\xi(0))+K_1T+
K_2\mathbb{E}\int_{0}^{\eta_{\Delta}\wedge t_1} \vert \bar{x}_{\Delta}(s)-x_{\Delta}(s)\vert ds\\
&+K_3\mathbb{E}\int_{-\tau}^{0}|\xi([s/\Delta]\Delta)-\xi(s)|ds+K_3\mathbb{E}\int_{0}^{\eta_{\Delta}\wedge t_1}\vert \bar{x}_{\Delta}(s)-x_{\Delta}(s)\vert ds\\
&\le  V(\xi(0))+K_1T+K_3\int_{-\tau}^0\mathbb{E}|\xi([s/\Delta]\Delta)-\xi(s)|ds\\
&+(K_2+K_3)\int_{0}^T(\mathbb{E}\vert \bar{x}_{\Delta}(s)-x_{\Delta}(s)\vert^p)^{1/p}ds\\
&\le V(\xi(0))+K_1T+K_3D\Delta^{\ell}+c_p(K_2+K_3)\Delta^{1/2}\psi(\Delta)T.
\end{align*}
This implies that 
\begin{equation*}
 \mathbb{P}(\eta_{\Delta}\leq T)\leq \frac{V(\xi(0))+K_1T+K_3D\Delta^{\ell}+c_p(K_2+K_3)\Delta^{1/2}\psi(\Delta)T}{V(1/k)\wedge V(k)},
\end{equation*}
as required.
\end{proof}
\section{Convergence analysis}\label{sect5}
In this section, we study the  finite-time convergence of the TEM solutions to the true solution of SDDE \eqref{cir4}. Further, we show that the TEM solutions converge  to the true solution of SDDE \eqref{cir4} in probability. We perform simulation examples to support the findings and justify the convergence result for valuing some option contracts.
\subsection{Finite-time error bound}
The following lemma shows that the TEM solutions converge to the true solution of SDDE \eqref{cir4} in finite time. 
\begin{lemma}\label{eq:l8}
Let Assumptions \ref{sec2:assump:1} and \ref{sec2:assump:2} hold. Then for any $p\geq 2$, fixed $T> 0$, sufficiently large $k>0$ and $\Delta\in (0,\Delta^*]$,  we have
\begin{equation}\label{eq:37}
\mathbb{E}\Big( \sup_{0\leq t \leq T}|x_{\Delta}(t\wedge\eta^{\Delta}_k)-x(t \wedge \eta^{\Delta}_k)|^p  \Big)\le K_8(\Delta^{\ell}\vee\Delta^{\frac{p}{2}}(\psi(\Delta))^p),
\end{equation}
where $K_8$ is a generic constant that depends on $k$ but is independent of $\Delta$ and $\eta^{\Delta}_k=\eta_k\wedge \eta_{\Delta}$, where $\eta_k$ and $\eta_{\Delta}$ are defined in \eqref{stoptime1} and \eqref{sec5:eq:tau} respectively.  Consequently, we have
\begin{equation}\label{eq:38}
\lim_{\Delta\rightarrow 0}\mathbb{E}\Big( \sup_{0\leq t \leq T}|\bar{x}_{\Delta}(t \wedge \eta^{\Delta}_k)-x(t \wedge \eta^{\Delta}_k)|^p \Big)=0.
\end{equation}
\end{lemma}
\begin{proof}
It follows from \eqref{eq2:1} and \eqref{eq:30} that for $t\in [0,t_1]$, we have
\begin{align*}
\mathbb{E}\Big(\sup_{0\leq t \leq t_1}|x_{\Delta}(t\wedge\eta^{\Delta}_k)-x(t\wedge\eta^{\Delta}_k)|^p\Big)\le K_4+K_5,
\end{align*}
 where
\begin{align*} 
K_4&=2^{p-1}\Big( \mathbb{E}\Big|\int_{0}^{t_1\wedge \eta^{\Delta}_k}(f_{\Delta}(\bar{x}_{\Delta}(s))-f(x(s)))ds\Big|^p\Big)\\
K_5&=2^{p-1}\Big(\mathbb{E}(\sup_{0\leq t \leq t_1}\Big|\int_{0}^{t_1\wedge \eta^{\Delta}_k)}(g_{\Delta}(\bar{x}_{\Delta}(s),\bar{x}_{\Delta}(s-\tau))-g(x(s),x(s-\tau)))dB(s)\Big|^p)\Big).
\end{align*}
By the H\"older inequality and \eqref{sec5:eq:recall0}, we compute
\begin{align*} 
K_4&\le 2^{p-1}T^{p-1}\Big( \mathbb{E}\int_{0}^{t_1\wedge \eta^{\Delta}_k}|f_{\Delta}(\bar{x}_{\Delta}(s))-f(x(s))|^pds\Big)\\
&\le 2^{p-1}T^{p-1}\Big(\mathbb{E}\int_{0}^{t_1\wedge \eta^{\Delta}_k}|f(\bar{x}_{\Delta}(s))-f(x(s))|^pds\Big)\\
&\le 2^{p-1}T^{p-1}G_k^p\mathbb{E}\int_{0}^{t_1\wedge \eta^{\Delta}_k}|\bar{x}_{\Delta}(s)-x(s)|^pds.
\end{align*}
Moreover, by the elementary inequality, we now have 
\begin{align}\label{grown0}
K_4&\le c_0\mathbb{E}\int_{0}^{t_1\wedge \eta^{\Delta}_k}|\bar{x}_{\Delta}(s)-x_{\Delta}(s)|^pds+c_0\mathbb{E}\int_{0}^{t_1\wedge \eta^{\Delta}_k}|x_{\Delta}(s)-x(s)|^pds\nonumber\\
&\le c_0\int_{0}^{T}\mathbb{E}|\bar{x}_{\Delta}(s)-x_{\Delta}(s)|^pds+c_0\int_{0}^{t_1}\mathbb{E}\Big(\sup_{0\le t\le s}|x_{\Delta}(t\wedge \eta^{\Delta}_k)-x(t\wedge \eta^{\Delta}_k)|^p\Big)ds,
\end{align}
where $c_0=2^{2(p-1)}T^{p-1}G_k^p$. By the Burkholder-Davis-Gundy inequality, Lemma \ref{lemma1} and \eqref{sec5:eq:recall0}, we also have 
\begin{align*}
K_5&\le 2^{p-1}T^{\frac{p-2}{2}} \bar{c}_p\Big(\mathbb{E}\int_{0}^{t_1\wedge \eta^{\Delta}_k}\vert g_{\Delta}(\bar{x}_{\Delta}(s),\bar{x}_{\Delta}(s-\tau))-g(x(s),x(s-\tau))\vert^pds \Big)\\
&\le 2^{p-1}T^{\frac{p-2}{2}} \bar{c}_p\Big(\mathbb{E}\int_{0}^{t_1\wedge \eta^{\Delta}_k}\vert g(\bar{x}_{\Delta}(s),\bar{x}_{\Delta}(s-\tau))-g(x(s),x(s-\tau))\vert^pds \Big)
\end{align*}
where $\bar{c}_p$ is a positive constant. By the elementary inequality and Lemma \ref{lemma1}, we get
\begin{align*}
K_5&\le 2^{2(p-1)}T^{\frac{p-2}{2}} \bar{c}_p\Big(\mathbb{E}\int_{0}^{t_1\wedge \eta^{\Delta}_k}\vert g(\bar{x}_{\Delta}(s),\bar{x}_{\Delta}(s-\tau))-g(x_{\Delta}(s),x_{\Delta}(s-\tau))\vert^pds\Big)\\
&+ 2^{2(p-1)}T^{\frac{p-2}{2}} \bar{c}_p\Big(\mathbb{E}\int_{0}^{t_1\wedge \eta^{\Delta}_k}\vert  g(x_{\Delta}(s),x_{\Delta}(s-\tau))-g(x(s),x(s-\tau))\vert^pds\Big)\\
&\le c_1\Big(\mathbb{E}\int_{0}^{t_1\wedge \eta^{\Delta}_k}\vert \bar{x}_{\Delta}(s)-x_{\Delta}(s)\vert^pds\Big)+ c_1\Big(\mathbb{E}\int_{0}^{t_1\wedge \eta^{\Delta}_k}\vert \bar{x}_{\Delta}(s-\tau)-x_{\Delta}(s-\tau)\vert^pds\Big)\\
&+c_1\Big(\mathbb{E}\int_{0}^{t_1\wedge \eta^{\Delta}_k}\vert x_{\Delta}(s)-x(s)\vert^pds\Big)+c_1\Big(\mathbb{E}\int_{0}^{t_1\wedge \eta^{\Delta}_k}\vert x_{\Delta}(s-\tau)-x(s-\tau)\vert^pds\Big),
\end{align*}
where $c_1=2^{3(p-1)}T^{\frac{p-2}{2}} \bar{c}_pG^p_k$. This also means that 
\begin{align}\label{grown1}
K_5&\le c_1\int_{0}^T\mathbb{E}\vert \bar{x}_{\Delta}(s)-x_{\Delta}(s)\vert^pds+c_1\int_{-\tau}^0\mathbb{E}|\xi([s/\Delta]\Delta)-\xi(s)|^pds\nonumber\\
&+c_1\int_{0}^T\mathbb{E}\vert \bar{x}_{\Delta}(s)-x_{\Delta}(s)\vert^pds+ c_1\Big(\mathbb{E}\int_{0}^{t_1\wedge \eta^{\Delta}_k}\vert x_{\Delta}(s)-x(s)\vert^pds\Big)\nonumber\\
&+c_1\int_{-\tau}^0\mathbb{E}|\xi([s/\Delta]\Delta)-\xi(s)|ds+c_1\Big(\mathbb{E}\int_{0}^{t_1\wedge \eta^{\Delta}_k}\vert x_{\Delta}(s)-x(s)\vert^pds\Big)\nonumber\\
&\le 2c_1\int_{-\tau}^0\mathbb{E}|\xi([s/\Delta]\Delta)-\xi(s)|^pds
+2c_1\int_{0}^T\mathbb{E}\vert \bar{x}_{\Delta}(s)-x_{\Delta}(s)\vert^pds\nonumber\\
&+2c_1 \int_{0}^{t_1}\mathbb{E}\Big(\sup_{0\le t\le s}|x_{\Delta}(t\wedge \eta^{\Delta}_k)-x(t\wedge \eta^{\Delta}_k)|^p\Big)ds.
\end{align}
By combining $K_4$ and $K_5$, that is \eqref{grown0} and \eqref{grown1}, we now have
\begin{align*}
&\mathbb{E}\Big( \sup_{0\leq t \leq T}|x_{\Delta}(t\wedge\eta^{\Delta}_k)-x(t \wedge \eta^{\Delta}_k)|^p  \Big)\le 2c_1\int_{-\tau}^0\mathbb{E}|\xi([s/\Delta]\Delta)-\xi(s)|^pds\\
&+(c_0+2c_1)\int_{0}^{T}\mathbb{E}|\bar{x}_{\Delta}(s)-x_{\Delta}(s)|^pds+(c_0+2c_1)\int_{0}^{t_1}\mathbb{E}\Big(\sup_{0\le t\le s}|x_{\Delta}(t\wedge \eta^{\Delta}_k)-x(t\wedge \eta^{\Delta}_k)|^p\Big)ds.
\end{align*}
So by  Assumption \ref{sec2:assump:1}, Lemma \ref{sec5:eq:L1} and the Gronwall inequality, we obtain the required assertion as
\begin{align*}
&\mathbb{E}\Big( \sup_{0\leq t \leq T}|x_{\Delta}(t\wedge\eta^{\Delta}_k)-x(t \wedge \eta^{\Delta}_k)|^p  \Big)\\
&\le K_6(\Delta^{\ell}\vee\Delta^{\frac{p}{2}}(\psi(\Delta))^p)+K_7\int_{0}^{t_1}\mathbb{E}\Big(\sup_{0\le t\le s}|x_{\Delta}(t\wedge \eta^{\Delta}_k)-x(t\wedge \eta^{\Delta}_k)|^p\Big)ds\\
&\le  K_8(\Delta^{\ell}\vee\Delta^{\frac{p}{2}}(\psi(\Delta))^p),
\end{align*}
where $K_6=2c_1D+c_p(c_0+2c_1)$, $K_7=c_0+2c_1$ and $K_8=K_6e^{K_7}$. Moreover,  by Lemma \ref{sec5:eq:L1}, we also get \eqref{eq:38} by letting $\Delta\rightarrow 0$.
\end{proof}

\subsection{Convergence in probability}
The following theorem shows that the TEM solutions converge to the true solution of SDDE \eqref{eq2:1} in probability.
\begin{theorem}\label{final}
Let  $x(t)$ and $x_{\Delta}(t)$ be the true solution and the truncated EM solution of \eqref{eq2:1} and \eqref{eq:30} respectively. Then for any fixed $T>0$, $\Delta\in (0,\Delta^*]$ and $p\ge 2$, we have
\begin{equation}\label{converge1}
\lim_{\Delta\rightarrow 0}\Big( \sup_{0\leq t \leq T}|x_{\Delta}(t)-x(t)|^p\Big)=0\text{ in probability}.
\end{equation}
and consequently 
\begin{equation}\label{converge2}
\lim_{\Delta\rightarrow 0}\Big( \sup_{0\leq t \leq T}|\bar{x}_{\Delta}(t)-x(t)|^p\Big)=0\text{ in probability},
\end{equation}
where $\bar{x}_{\Delta}(t)$ is defined in \eqref{TEM1}.
\end{theorem}
\begin{proof}
For arbitrarily small constants $\epsilon$ and $\lambda$, set
\begin{equation*}
\bar{\Omega}=\Big\{\omega: \sup_{0\leq t \leq T}|x_{\Delta}(t)-x(t)|^p\ge \lambda\Big\}.
\end{equation*}
Then
\begin{align*}
\lambda\mathbb{P}(\bar{\Omega}\cap(\eta^{\Delta}_k\ge T))&=\lambda\mathbb{E}\Big( 1_{(\eta^{\Delta}_k\ge T)} 1_{\bar{\Omega}}\Big)\\
&\le \mathbb{E}\Big( 1_{(\eta^{\Delta}_k\ge T)} \sup_{0\leq t \leq T}|x_{\Delta}(t)-x(t)|^p\Big)\\
&\le \mathbb{E}\Big(\sup_{0\leq t \leq T\wedge \eta^{\Delta}_k}|x_{\Delta}(t)-x(t)|^p\Big)\\
&\le \mathbb{E}\Big(\sup_{0\leq t \leq T}|x_{\Delta}(t\wedge \eta^{\Delta}_k)-x(t\wedge \eta^{\Delta}_k)|^p\Big). 
\end{align*}
By Lemma \ref{eq:l8}, we get
\begin{align}\label{prob3}
\mathbb{P}(\bar{\Omega}\cap(\eta^{\Delta}_k\ge T))\le \frac{K_8(\Delta^{\ell}\vee\Delta^{\frac{p}{2}}(\psi(\Delta))^p)}{\lambda}.
\end{align}
Furthermore, we compute
\begin{align}\label{prob4}
\mathbb{P}(\bar{\Omega})&\le \mathbb{P}(\bar{\Omega}\cap(\eta^{\Delta}_k\ge T))+\mathbb{P}(\eta^{\Delta}_k\le T)\nonumber\\
&\le \mathbb{P}(\bar{\Omega}\cap(\eta^{\Delta}_k\ge T))+\mathbb{P}(\eta_k\le T)+\mathbb{P}(\eta_{\Delta}\le T).
\end{align} 
So, by substituting \eqref{probability1}, \eqref{probability2} and \eqref{prob3} into \eqref{prob4}, we have
\begin{align}
\mathbb{P}(\bar{\Omega})&\le \frac{V(\xi(0))+K_0T}{V(k)\wedge V(1/k)}+ \frac{K_8(\Delta^{\ell}\vee\Delta^{\frac{p}{2}}(\psi(\Delta))^p)}{\lambda}\nonumber\\
&+\frac{V(\xi(0))+K_1T+K_3D\Delta^{\ell}+c_p(K_2+K_3)\Delta^{1/2}\psi(\Delta)T}{V(1/k)\wedge V(k)}.
\end{align}
Therefore, we can select $k$ sufficiently large such that
\begin{equation}\label{epp1}
\frac{2V(\xi(0))+K_0T+K_1T}{V(k)\wedge V(1/k)}< \frac{\epsilon}{2}
\end{equation}
and select $\Delta$ so small such that
\begin{align}\label{epp2}
\frac{K_3D\Delta^{\ell}+c_p(K_2+K_3)\Delta^{1/2}\psi(\Delta)T}{V(1/k)\wedge V(k)}+\frac{K_8(\Delta^{\ell}\vee\Delta^{\frac{p}{2}}(\psi(\Delta))^p)}{\lambda}< \frac{\epsilon}{2}.
\end{align}
So by combining \eqref{epp1} and \eqref{epp2}, we now have 
\begin{equation}
\mathbb{P}\Big(\sup_{0\leq t \leq T}|x_{\Delta}(t)-x(t)|^p\ge \lambda\Big)< \epsilon,
\end{equation}
as desired. However, by Lemma \ref{sec5:eq:L1}, we also obtain \eqref{converge2} by setting $\Delta\rightarrow 0$.
\end{proof}
\subsection{Numerical simulation}
In this illustrative simulation example, we compare the performance of the truncated EM method (TEM) constructed for SDDE \eqref{eq2:1} with the backward EM method (BEM). We should clarify that, to the best of our knowledge, there exist no relevant literature for the numerical treatment of SDDE \eqref{eq2:1} based on the backward  EM method. This illustration is just for the purpose of comparison. For the sake of simplicity, let us consider the following form of SDDE \eqref{eq2:1} given by
\begin{equation}\label{eq:sm1}
 dx(t)=4(2-x(t)^2)dt+0.5x(t-2)^{2/3}x(t)^{3/5}dB(t),
\end{equation}
with the initial data $\xi(0)=0.2$, where  $\tau=2$, $\gamma=2$, $r=2/3$ and $\theta=3/5$. Clearly, we see that Assumption \ref{sec2:assump:1} is satisfied. Moreover, we note that
\begin{equation*}
\sup_{\vert x\vert\vee\vert y\vert \le u}\Big(|f(x)|\vee g(x,y)\Big)\le 6.5u^2
\end{equation*}
for all $u\ge 1$. This means that we have $z(u)=6.5u^2$ with inverse $z^{-1}(u)=(u/6.5)^{1/2}$. If we choose $\psi(\Delta)=\Delta^{-2/3}$, then 
\begin{align*}
z^{-1}(\psi(\Delta))=\Big(\frac{\Delta^{-2/3}}{6.5}\Big)^{1/2}.
\end{align*}  
In the following, we present numerical experiments illustrating stability, sample path behaviour, convergence and delay sensitivity of the TEM method.

\vspace{1em}

\begin{itemize}
\item 
\begin{figure}[!htbp]
  \centerline{\includegraphics[scale=0.8]{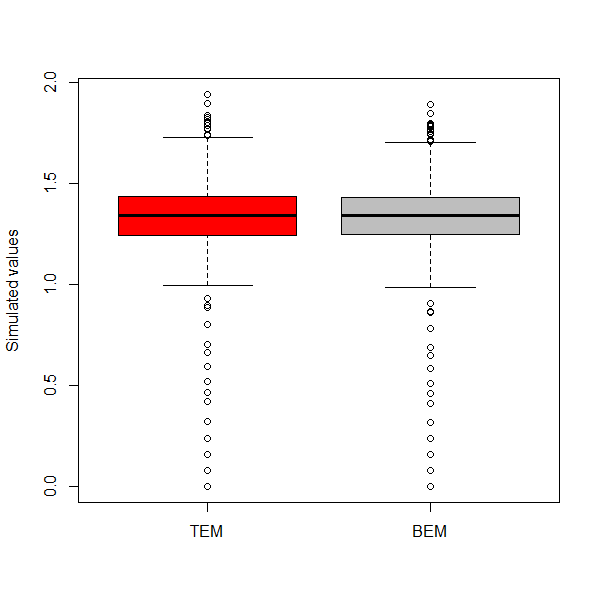}}
  \caption{Empirical distribution of the TEM and BEM solutions}
  \label{Fig:figure2}
\end{figure}

\vspace{1em}

\begin{table}[!htbp]
\centering
\begin{tabular}{lcccccc}
\hline
Method & Min & Mean & SD & Kurtosis & Skewness & Max \\
\hline 
TEM & 0.0000 & 1.3720 & 0.1880 & 14.5202 & -1.7707 & 1.9380 \\
BEM & 0.0000 & 1.3720 & 0.1825 & 12.9464 & -2.0372 & 1.9010 \\
\hline
\end{tabular}
\caption{Statistics of empirical distribution of the TEM and BEM solutions}
\label{Tab:table1}
\end{table}
Figure \ref{Fig:figure2} and Table \ref{Tab:table1} present the empirical distribution of the numerical solutions obtained from the TEM and BEM methods using $\Delta=10^{-2}$. It is observed that both methods yield identical mean values, indicating strong agreement in terms of first-order moments and providing evidence of weak convergence. The standard deviations are also close, suggesting that both schemes capture the variability of the solution accurately. Furthermore, both distributions exhibit negative skewness, indicating a tendency toward lower values, and high kurtosis, reflecting the presence of heavy tails. These features are consistent with the nonlinear and stochastic nature of the underlying model. Overall, the TEM method produces results that are statistically comparable to those of the BEM method, while maintaining stability and robustness.

\vspace{0.1em}

\item
\begin{figure}[!htbp]
  \centerline{\includegraphics[scale=0.8]{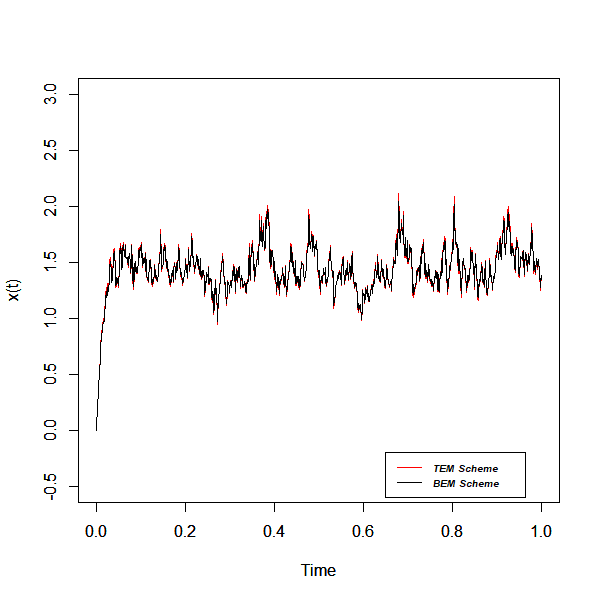}}
  \caption{Convergence of the TEM and BEM solutions}
  \label{Fig:figure1}
\end{figure}
Figure \ref{Fig:figure1} compares sample paths obtained using the TEM and BEM methods over the time interval $[0,1]$  using $\Delta=10^{-2}$. The two numerical solutions exhibit close agreement throughout the simulation, indicating that both methods approximate the same underlying SDDE \ref{eq:sm1}. Following a short initial transient, the solutions remain stable and fluctuate around a steady level without numerical divergence. The TEM solution displays slightly higher local variability, with marginally sharper peaks, reflecting its explicit nature despite the application of coefficient truncation. In contrast, the BEM solution appears smoother due to the implicit treatment of the drift term, which introduces additional numerical damping. However, the discrepancy between the two trajectories remains small across the entire time interval. Overall, the strong agreement between the TEM and BEM schemes demonstrates that truncation effectively stabilises the explicit Euler-Maruyama method, producing numerical behaviour comparable to that of the implicit BEM scheme. These results provide empirical support for the stability and convergence properties of the TEM method when applied to SDDEs with superlinear coefficients.

\vspace{0.1em}

\item
\begin{figure}[!htbp]
  \centerline{\includegraphics[scale=0.8]{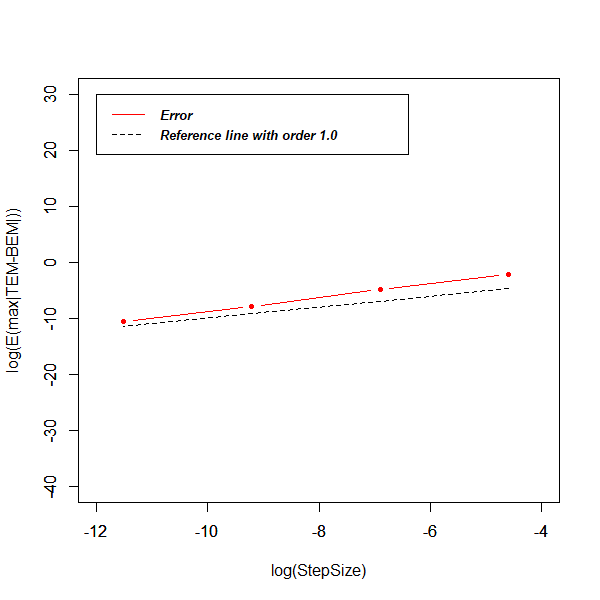}}
  \caption{Errors between the TEM and BEM solutions}
  \label{Fig:figure3}
\end{figure}

The log-log error plot in Figure \ref{Fig:figure3} illustrates the convergence behaviour of the TEM method  using the step sizes $10^{-2}$, $10^{-3}$, $10^{-4}$ and $10^{-5}$. The error is measured as the expected maximum difference between the TEM and BEM solutions. From the figure, the slope of the error curve is approximately equal to one, as it closely follows the reference line of order $1.0$. This indicates that the TEM method achieves first-order convergence with respect to the step size $\Delta$. This result confirms the theoretical convergence properties of the truncated EM scheme, even in the presence of delay and non-Lipschitz coefficients.

\newpage

\item
\begin{figure}[!htbp]
  \centerline{\includegraphics[scale=0.8]{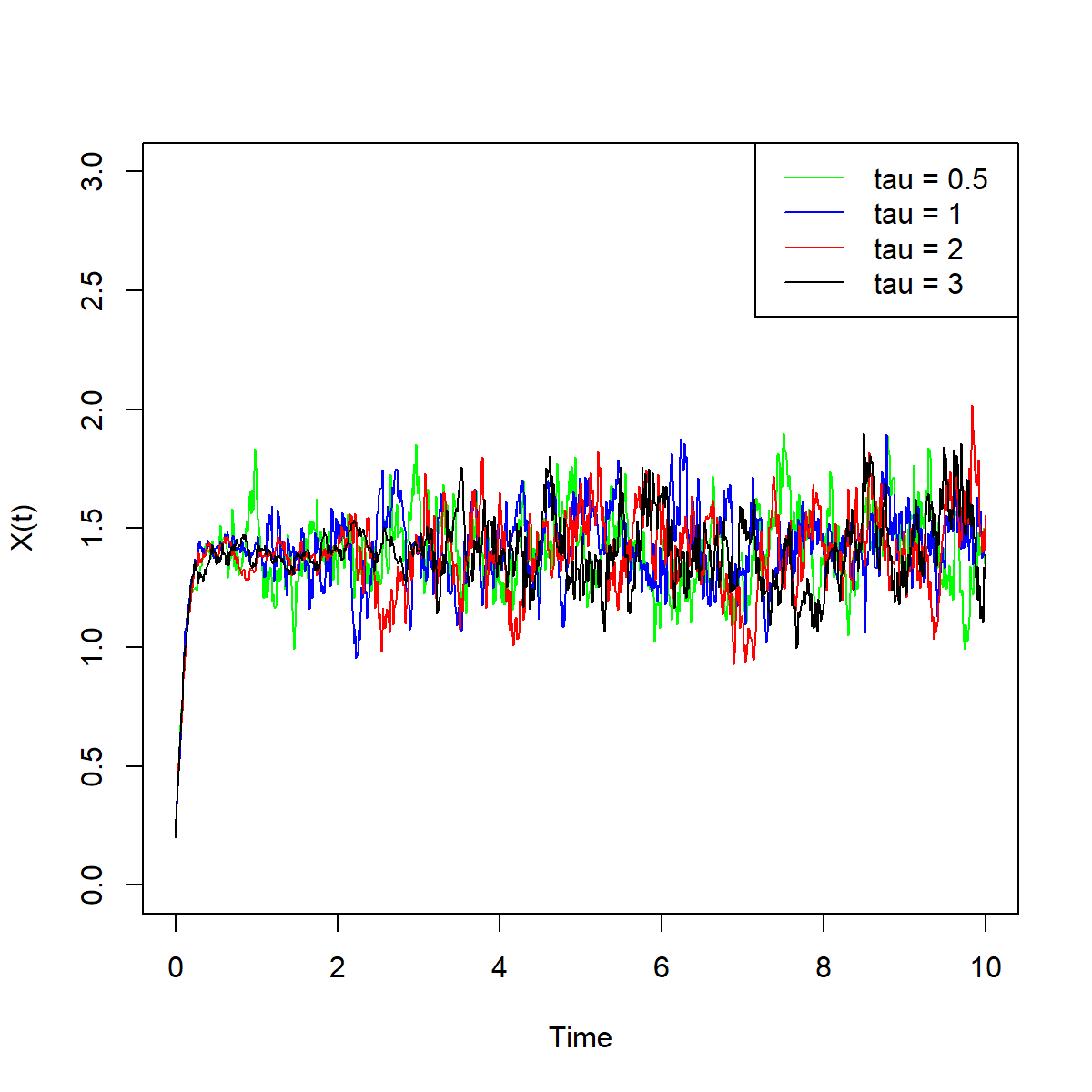}}
  \caption{Analysis of the delay sensitivity}
  \label{Fig:figure4}
\end{figure}
\noindent Figure \ref{Fig:figure4} illustrates the effect of varying the delay parameter $\tau$ (i.e, $\tau=0.5, 1, 2, 3$) on the numerical solution obtained via the TEM method. It is observed that all trajectories remain stable and bounded, confirming the robustness of the scheme in handling nonlinear and non-Lipschitz coefficients. Furthermore, the solutions exhibit only mild sensitivity to changes in $\tau$. While larger delay values introduce slightly increased variability due to stronger memory effects, the overall qualitative behaviour of the solution remains consistent. This indicates that, for the chosen parameter set, the system is relatively insensitive to delay variations within the considered range. These results highlight the effectiveness of the TEM method in preserving stability even in the presence of delay and super-linear diffusion terms.
\end{itemize}

\subsection{Application in finance}
We illustrate the convergence of the TEM method in a financial setting by considering bond pricing and path-dependent option valuation via Monte Carlo simulation.
\begin{lemma}
Let $x(t)$ and $\bar{x}_{\Delta}(t)$ denote the exact solution of \eqref{eq2:1} and the truncated EM approximation given in \eqref{eq:30}, respectively. Consider a zero-coupon bond with maturity $T$, whose price is given by
\begin{equation}\label{eq:52}
B(T) = \mathbb{E}\left[\exp\left(-\int_{0}^{T} x(t)\,dt\right)\right].
\end{equation}
A natural Monte Carlo approximation based on the numerical scheme is then defined by
\begin{equation}\label{eq:52*}
B_{\Delta}(T) = \mathbb{E}\left[\exp\left(-\int_{0}^{T} \bar{x}_{\Delta}(t)\,dt\right)\right].
\end{equation}
Then, as a consequence of Theorem \ref{final}, we have
\begin{equation*}
\lim_{\Delta \to 0} \left| B_{\Delta}(T) - B(T) \right| = 0.
\end{equation*}
\end{lemma}
\begin{proof}
Let $\epsilon, \delta\in (0,1)$ be arbitrarily small. It is sufficient to prove that 
\begin{align*}
\mathbb{P}\Big[\Big\vert \exp\Big(-\int_{0}^{T}x(t)dt\Big)- \exp\Big(-\int_{0}^{T}\bar{x}_{\Delta}(t)dt\Big) \Big\vert\ge \delta\Big]<\epsilon.
\end{align*}
Using the inequality $ \exp(-\vert x\vert)-\exp(-\vert y\vert)\le \vert x-y \vert$, we have
\begin{align*}
\Big\vert \exp\Big(-\int_{0}^{T}x(t)dt\Big)- \exp\Big(-\int_{0}^{T}\bar{x}_{\Delta}(t)dt\Big) \Big\vert&\le \Big\vert \int_{0}^{T}[x(t)-\bar{x}_{\Delta}(t)]dt \Big\vert\\
&\le T\sup_{0\le t\le T}\vert x(t)-\bar{x}_{\Delta}(t)\vert.
\end{align*}
By applying  Theorem \ref{final}, we obtain the desired assertion.
\end{proof}

\begin{lemma}
Let $x(t)$ and $\bar{x}_{\Delta}(t)$ denote the exact solution of \eqref{eq2:1} and its truncated Euler-Maruyama approximation given in \eqref{eq:30}, respectively. Consider a fixed-strike lookback put option with strike $K$, whose expected payoff is defined by
\begin{equation}\label{eq:52**}
P = \mathbb{E}\left[\left(K - \inf_{0 \le t \le T} x(t)\right)^+\right].
\end{equation}
A corresponding numerical approximation based on the truncated EM scheme is given by
\begin{equation}\label{eq:52***}
P_{\Delta} = \mathbb{E}\left[\left(K - \inf_{0 \le t \le T} \bar{x}_{\Delta}(t)\right)^+\right].
\end{equation}
Then, as a consequence of Theorem \ref{final}, we have
\begin{equation*}
\lim_{\Delta \to 0} |P - P_{\Delta}| = 0.
\end{equation*}
\end{lemma}

\begin{proof}
In other words, we need to prove that
\begin{align*}
\lim_{\Delta\rightarrow 0}\vert  (K-\inf_{0\le t\le T}x(t))^+-(K-\inf_{0\le t\le T}\vert \bar{x}_{\Delta}(t)\vert)^+\vert=0\quad \text{in probability}.
\end{align*}
This also means that the theorem holds as long as we can establish that for any small constants $\epsilon>0$ and $\delta\in (0,1)$
\begin{align}\label{look1}
\mathbb{P}(\vert (K-\inf_{0\le t\le T}x(t))^+-(K-\inf_{0\le t\le T}\vert \bar{x}_{\Delta}(t)\vert)^+\vert\ge \delta)<\epsilon
\end{align}
holds for all sufficiently small $\Delta$. We observe that
\begin{align}
\vert  (K-\inf_{0\le t\le T}x(t))^+-(K-\inf_{0\le t\le T}\vert \bar{x}_{\Delta}(t)\vert)^+\vert
&\le \vert \inf_{0\le t\le T}x(t)-\inf_{0\le t\le T}\vert \bar{x}_{\Delta}(t)\vert\vert\nonumber\\
&\le \sup_{0\le t\le T}\vert x(t)-\vert \bar{x}_{\Delta}(t)\vert\vert\nonumber\\
&\le \sup_{0\le t\le T}\vert x(t)-\bar{x}_{\Delta}(t)\vert.
\end{align}
Then, it follows that 
\begin{align}\label{look2}
\mathbb{P}(\vert (K-\inf_{0\le t\le T}x(t))^+-(K-\inf_{0\le t\le T}\vert \bar{x}_{\Delta}(t)\vert)^+\vert\ge \delta)
\le \mathbb{P}(\sup_{0\le t\le T}\vert x(t)-\bar{x}_{\Delta}(t)\vert\ge \delta).
\end{align}
So, by Theorem \ref{final}, we now have
\begin{align}\label{look3}
\mathbb{P}(\sup_{0\le t\le T}\vert x(t)-\bar{x}_{\Delta}(t)\vert\ge \delta)<\epsilon
\end{align}
for all sufficiently small $\Delta$. So by combining \eqref{look2} and \eqref{look3} gives us \eqref{look1}.
\end{proof}

\begin{example}
We investigate the numerical performance of the TEM method in financial applications. In particular, we consider two derivative products: a zero-coupon bond and a fixed-strike lookback put option, both governed by the SDDE \eqref{eq:sm1}. The numerical experiments are performed with maturity time $T = 5$ and strike price $K = 1$. We implement the TEM scheme using decreasing step sizes $\Delta = 10^{-2}, 10^{-3}, 10^{-4}, 10^{-5}$ in order to examine convergence behaviour. For each configuration, Monte Carlo simulations with $MC = 2000$ sample paths are used to approximate the expected zero-coupon bond and fixed-strike lookback put option values.

\begin{itemize}
\item Zero-coupon bond pricing: The price of a zero-coupon bond with maturity $T$ is given by \eqref{eq:52}, where $x(t)$ represents the short rate process governed by the SDDE \eqref{eq:sm1}. Using the TEM approximation $\bar{x}_{\Delta}(t)$, the numerical bond price is given by \eqref{eq:52*}. The integral is approximated using a Riemann sum becomes
\begin{align*}
\int_0^T \bar{x}_{\Delta}(t)\,dt \approx \sum_{k=0}^{N-1} X_{\Delta}(t_k)\Delta.
\end{align*}
\item Lookback put option pricing: Consider a fixed-strike lookback put option with strike $K$. Its payoff is defined by \eqref{eq:52**}. The corresponding numerical approximation is given by \eqref{eq:52***}.
\item Monte Carlo implementation: For each step size $\Delta$, we simulate $M$ independent sample paths of the TEM scheme by
\begin{align*}
\{X_{\Delta}^{(i)}(t_k)\}_{i=1}^M.
\end{align*}
The Monte Carlo estimators for \eqref{eq:52*} and \eqref{eq:52***} are given by
\begin{align*}
B_{\Delta}^{MC}(T) &= \frac{1}{M} \sum_{i=1}^M 
\exp\left(-\sum_{k=0}^{N-1} X_{\Delta}^{(i)}(t_k)\Delta\right), \\
P_{\Delta}^{MC} &= \frac{1}{M} \sum_{i=1}^M 
\left(K - \min_{0 \le k \le N} X_{\Delta}^{(i)}(t_k)\right)^+.
\end{align*}
\item Numerical results: The Monte Carlo estimates for the bond price and the fixed-strike lookback put option are reported in Table \ref{Tab:table5} for different step sizes $\Delta$. From the table, it is observed that the bond price estimates remain stable across all step sizes. In particular, the values fluctuate only slightly around $0.001$, indicating strong numerical consistency and supporting the convergence of the truncated Euler-Maruyama method. In contrast, the lookback put option exhibits a different behaviour. For moderate step sizes $\Delta = 10^{-2}$ and $10^{-3}$, the estimated values remain stable around $0.8$. However, as the step size decreases further, slight variations are observed, and a noticeable increase occurs at $\Delta = 10^{-5}$. This behaviour can be attributed to the path-dependent nature of the lookback payoff, which depends on the infimum of the process. As $\Delta$ decreases, the time discretisation becomes finer, increasing the likelihood of capturing lower values of the trajectory. Consequently, the numerical approximation of the infimum becomes more sensitive, leading to fluctuations in the estimated payoff. Moreover, the observed variation at very small step sizes is influenced by Monte Carlo sampling error. Since the payoff depends on extreme values of the sample paths, a larger number of simulations is typically required to obtain stable estimates. Overall, the results confirm that the truncated Euler-Maruyama method provides stable and reliable approximations for bond pricing, while the lookback option highlights the increased sensitivity associated with path-dependent financial derivatives.

\vspace{0.1em}

\begin{table}[!htbp]
\centering
\begin{tabular}{ccc}
\hline
$\Delta$ & Bond Price $B_{\Delta}(T)$ & Lookback Put $P_{\Delta}$ \\
\hline
$10^{-2}$ & 0.0009993 & 0.8000000 \\
$10^{-3}$ & 0.0010098 & 0.8000000 \\
$10^{-4}$ & 0.0010066 & 0.8000462 \\
$10^{-5}$ & 0.0010109 & 1.0000000 \\
\hline
\end{tabular}
\caption{Monte Carlo estimates for bond price and lookback put option}
\label{Tab:table5}
\end{table}
\end{itemize}
\end{example}

\section*{Acknowledgement}
The author would like to acknowledge the financial support from the Heilbronn Institute for Mathematical Research (HIMR) and the UKRI/EPSRC Additional Funding Programme for Mathematical Sciences. 

\section*{Declarations}
\textbf{Conflict of interest:} The author declares that he has no conflict of interest.

\end{document}